\documentclass[12pt]{article}
\usepackage{amssymb, url, amsthm, amsmath}
\usepackage{amscd}
\newtheorem{theorem}{\noindent Theorem}
\newtheorem{lemma}{\noindent Lemma}
\newtheorem{definition}{\noindent Definition}
\newtheorem{corollary}{\noindent Corollary}

\date{}
\urldef\vershik\url{vershik@pdmi.ras.ru}
\author{A.~M.~Vershik\thanks{St.~Petersburg
Department of Steklov Institute of Mathematics.
 Email: \vershik. Partially supported by the RFBR grant 11-01-00677-a.}}

\date{}
\title{The Pascal automorphism has a continuous part in the spectrum.I}

\begin{document}

\maketitle

\rightline{\it \textbf{In memory of V.~I.~Arnold, a troublemaker}}
\bigskip
\bigskip
\begin{flushright}
\parbox{12cm}
{Mathematicians who are only mathematicians have exact minds, provided
all things are explained to them by means of definitions and axioms;
otherwise they are inaccurate and insufferable, for they are only
right when the principles are quite clear.}

\vskip2mm
Blaise Pascal, {\it Thoughts}, English translation by W.~F.~Trotter
\end{flushright}
 \bigskip
 \bigskip

\noindent\textbf{Dedication.}
Dima Arnold (as well as me) was very fond of B.~Pascal, and disliked
R.~Descartes, seeing him as a forerunner of Bourbakism he
hated so much. As to me, in my youth I had great respect for Bourbaki,
a very high opinion of his 5th volume, and even once wrote to him (N.~Bourbaki) a
long eulogistic letter, of which Dima did not approve.
In reply, N.~Bourbaki (impersonated by J.~Dieudonn\'e) presented me the
next volume {\it Integration}, which had just appeared; the
topic was close to my interests, but the volume turned out to
be a failure. I was distressed and started to believe that perhaps
Arnold was right.

The keen interest to combinatorics and asymptotic problems, which appeared
in the last years of V.~I.~Arnold's life and made us even closer, was, I
believe, another manifestation of the fact that his mind revolted at
any limits and
prohibitions, he always violated canons, or, better to say,
introduced new canons; and he was able to do this, because he was
(according to Pascal) not only a mathematician.

\newpage
\begin{abstract}
We describe the plan of the proof that the spectrum of the Pascal automorphism, a natural
transformation of the space of paths in the Pascal graph (= infinite
Pascal triangle), and other similar (adic) automorphisms have continuous spectrum.
If we realize this automorphism as the
shift in the space of $0-1$ sequences, we obtain
a stationary measure, called the Pascal measure, whose
properties we study. In particular, the set of almost periodic
sequences is of zero Pascal measure, and it is this fact that eventually
implies the existence of continuity part of the spectrum of the corresponding operator.

The transformations generated by classical graded graphs, such as the
ordinary and multidimensional Pascal graphs, the Young graph, the graph of
walks in Weyl chambers, etc., provide examples of combinatorial
nature from a new and very interesting class of adic transformations
introduced as early as in \cite{V81}; some considerations by
V.~I.~Ar\-nold also lead to such transformations. We discuss problems
arising in this field.
\end{abstract}

\tableofcontents

\newpage
\section{Introduction: the definition of the adic Pascal automorphism}

\subsection{A linear order and other structures on the set of vertices
of the cube}

The definition of the Pascal automorphism, which is an example of an adic
transformation, was given in
\cite{V81,V82}; in the same papers, the problem of finding
its spectrum was suggested. The definition is based on introducing a
lexicographic (linear) order on the set of paths of the finite or
infinite Pascal graph.\footnote{The term ``Pascal triangle'' is used
more widely, but it applies mainly to the finite object; at the
vertices of this graph one usually writes the binomial coefficients. Thus it
is not quite natural to call the infinite Pascal graph a triangle. But
there are also objections to applying the term ``Pascal triangle'' to
the finite graph: first of all, since the discovery of this graph is
attributed also to the Indian Pingala (10th century), the Persian Omar
Khayyam (12th century), and the Chinese Yang Hui (13th century). On
the other hand, the results of the Google search ``Pascal graph''
are related mainly to graphics in the programming language Pascal.}
Introducing this lexicographic order, in turn, reduces to
introducing a natural linear order on the set of all subsets of given
cardinality in the set of $n$ elements, or, in the geometric language,
on the set of vertices of the
$n$-dimensional cube with a given number of nonzero
coordinates. We begin with several versions of the definition of this order.

\medskip
\textbf{1}. Let $I_n=\{0,1\}^n$ be the set of vertices of the unit
$n$-dimensional cube. Consider the hyperplanes that contain vertices
with the sum of coordinates equal to
$m$, where $m=0,1,\dots, n$, and denote by $C_{m,n-m}$ the set of vertices on
the $m$th hyperplane; these vertices have $m$ coordinates equal to $0$
and $n-m$ coordinates equal to $1$. We define a linear order on all
$C_{n,m}$ by induction. For $n=2$, the order on the one-point sets
$C_{2,0}, C_{0,2}$ is trivial, and on $C_{1,1}$ it is as follows:
$(0,1)\succ (1,0)$. Now assume that the order is defined on
$C_{k, n-k}$ for all $k=0,1, \dots, n$; then on
$C_{m,n+1-m}$ it is defined by the following rule. Given a pair of
vertices with the same last coordinate, their order is the same as for the pair of
vertices from $C_{m,n-m}$ obtained by deleting the common last coordinate;
if the last coordinates are different, then the greater vertex is the
one for which this coordinate is equal to $1$.

\smallskip
\textbf{2}. The same definition in slightly different terms reads as
follows. It is convenient to denote  $k=n-m$. Consider the family
$C(m,k)$ of all $C_{m+k}^m$ subsets of cardinality $k$ in a linearly
ordered set of  $m+k$ elements
$\{\textbf{1,2,\dots, k, k+1, \dots, k+m}\}$, and introduce on this set a
linear lexicographic order by the following rule: a subset $F$ is greater than a subset $G$
if the maximum index of an element from $F$ that does not belong to
$G$ is greater that the maximum index of an element from $G$ that
does not belong to $F$. The smallest subset according to this order has
the characteristic function
$(\underbrace{1,{\ldots} ,1}_k,\underbrace{0,{\ldots} ,0}_m)$, and the
greatest one, $(\underbrace{0,{\ldots}
,0}_m, \underbrace{1,{\ldots},1}_k)$.
Thus we have linearly ordered all points of the sets $C(m+k,m)$.

\smallskip
\textbf{3}. Finally, the ``numerical'' interpretation of this order is
as follows. Consider the points of a finite-dimensional cube (with
ordered coordinates) as binary expansions of integers.
For two integers with the same number of ones (and zeros) in
this expansion, the order in question is just the ordinary order $<$ on the set of
positive integers.

\smallskip
The equivalence of all three independent definitions of the order is
easy to verify.

Since a path in the finite Pascal triangle of height $n$ can be
identified with a $0-1$ sequence, we have in fact linearly ordered
each set of paths leading to a finite vertex  of the triangle.

\subsection{The infinite-dimensional cube and the Pascal automorphism; Pascal cocycle.}

Now consider the set of vertices of the countable product of two-point sets
$I^{\infty}=\{0,1\}^{\infty}$; we will briefly call it the
infinite-dimensional cube. On this remarkable object, there is a huge number of important
mathematical structures, which have many useful interpretations and
various applications.

First of all, we regard the infinite-dimensional cube as the compact
additive group of dyadic integers
$\textbf{Z}_2$. We realize it as the group of sequences of residues
modulo $2^n$ and use the additive notation. In our interpretation, a
dyadic integer is also an infinite one-sided sequence of zeros and
ones. The weak topology on $I^\infty$ gives rise to the structure of a
standard measure space. The Bernoulli measure $\mu$ is the infinite
product of the measures with probabilities $(1/2,1/2)$ on the factors
$\{0,1\}$. It is simultaneously the
Haar measure on the cube regarded as the group
$\textbf{Z}_2$. One can also consider other Bernoulli measures
$\mu_p$, which are the products of the measures $(p,1-p)$, $0<p<1$; they are no
longer invariant (only quasiinvariant) under addition (but will be
invariant under the Pascal automorphism, see below). The
infinite-dimensional cube can be naturally identified with the set of
infinite paths in the Pascal graph (a zero corresponds to choosing the
left direction, and a one, to choosing the right direction). It is
this space that will be the phase space of the  Pascal automorphism
defined below. All these interpretations are identical, and the
measure is the same; the choice of a convenient realization is a
matter of taste. For us, it is usually convenient to use the
infinite-dimensional cube; the topology of the space is not of
importance.

It is convenient to write dyadic integers in the form
$$\underbrace{0,{\ldots} ,0}_{m_1}\underbrace{1,{\ldots} ,1}_{k_1}**= 0^{m_1}1^{k_1}** .$$

Obviously, the shift $T$ in the additive group
$\textbf{Z}_2$ defined by the formula
$Tx=x+1$ preserves the Haar--Bernoulli measure; in
dynamical systems, it is called the {\it odometer}, or the
{\it dyadic automorphism}. This is one of the simplest ergodic
automorphisms; its spectrum (= the set of eigenvalues) is the group of
all roots of unity of order
$2^n$, $n=1,2, \dots$. The orbits of the odometer are the cosets of the
(dense) subgroup ${\Bbb Z} \subset \textbf{Z}_2$. The general adic
model of measure-preserving transformations is a far-reaching
generalization of the odometer (see \cite{V81,V82}).

Consider the 2-adic metric $\rho$ on the additive group $\textbf{Z}_2$
of dyadic integers; it induces the weak topology on
$I^{\infty} $.  The metric $\rho$ looks as follows:
$\rho(x,y)=\|x-x'\|$, where $\|g\|=2^{-t(g)}$ is the {\it
canonical normalization}; here
$t(g)$ is the index of the first nonzero coordinate of $g$. The metric
$\rho$ is obviously invariant with respect to the odometer; however, as we
will see, it is not invariant with respect to the Pascal automorphism.

Now we introduce an order on
$\textbf{Z}_2 \sim I^{\infty}=\{0,1\}^{\infty}$ as follows. We say
that two vertices (points) of the infinite-dimensional cube are
comparable if their coordinates coincide from some index
(i.e., they have ``the same tail'') and the number of ones
before this index in both sequences is the same. Given two comparable
sequences, the greater one is, by definition, the sequence whose initial
finite segment is greater in the sense of the order defined above.
This is the lexicographic order on the infinite-dimensional cube we
need; denote it by $\prec$. One can easily extend all the
descriptions of the order on the finite-dimensional cube given above
to the infinite-dimensional cube.

But every vertex of the infinite-dimensional cube can be regarded as
an infinite path in the Pascal graph. In the order we have introduced, paths are comparable
if they have the same ``tail,'' i.e., their coordinates coincide from
some index. Thus we have defined a linear order on the set of paths in
the Pascal graph.

The order type of the class of comparable paths is that of the
unilaterally ($\Bbb N$) or bilaterally ($\Bbb Z$) infinite
chain; it is infinite to the left if the corresponding vertex of the
cube has only finitely many zeros, and infinite to the right if it has
finitely many ones. For all other points (paths), the order type is
$\Bbb Z$; they constitute a set of full Bernoulli measure. One may say
that we have redefined the order on the cosets of the subgroup $\Bbb
Z$; each coset breaks into countably many linearly ordered subsets.

\begin{definition}[\cite{V82}]
The Pascal automorphism  is the map $P$ of the
infinite-dimensional cube (in any interpretation) to itself that sends
every point to its immediate successor in the sense of the order
$\prec$ introduced above.
\end{definition}

The immediate successor, as well as
the immediate predecessor, exists for all points except for countably many,
hence the Pascal automorphism and its inverse are defined everywhere
except for countably many points (more exactly, except for the
elements of the group $\Bbb Z$
regarded as a subgroup in $\textbf{Z}_2$). It is easy to see that the
transformation $P$ is measurable and even continuous in the weak
topology everywhere except for the above-mentioned exclusions.

\begin{definition}
Suppose two points of the infinite dimensional cube $x=\{x_k\}$ and $y=\{y_k\}$ have the same tail:
$\exists K: x_k=y_k, k>K$
The function
$$c(x,y)=\#\{i:x_i=0,i<K\}-\#\{i:y_k=0,i<K\}$$
called a Pascal cocycle.
\end{definition}
It follows that $c(y,x)=-c(x,y)$ is difference of the numbers of 1-s in $x$ and $y$.
Pascal cocycle is not cohomologous to trivial cocycle on the infinite dimensional cube, but is trivial
on the finite dimensional cube.

\section{The Pascal automorphism as the result of a time change in the
odometer, and random substitutions on the group $\Bbb Z$}

Now we can proceed to a more detailed study of the Pascal adic
automorphism. Let us show that it preserves the
measure. Indeed, taking the immediate successor results in a substitution of finitely
many coordinates, so the measure-preserving property follows from the
fact that the Bernoulli measure with equal factors is invariant under
the action of the infinite symmetric group by substitutions of
coordinates. It is clear from above that every orbit of the Pascal
automorphism lies in one orbit of the odometer, namely, in the same
coset of $\Bbb Z$. Hence an element of the form $x-Px$ in the group of dyadic
integers lies in the subgroup
$\Bbb Z \subset \textbf{Z}_2 $, and, consequently, the Pascal
automorphism can be regarded as the result of a time change in the
odometer.

The orbit partition of the Pascal automorphism with respect to any Bernoulli measure
coincide ($\mod 0$  with orbit partition of the action of infinite symmetric group,
and consequently the list of ergodic invariant measures are the same.

If, using binary expansions, we identify (up to a set of zero measure)
the group $\textbf{Z}_2$ with the unit interval, then the Pascal
automorphism turns into a transformation of the interval which belongs
to the class of so-called
rational countable rearrangements.

Below we will present explicit formulas that show what time change
should be made in the odometer in order to obtain the Pascal
automorphism. As we will see, the Pascal automorphism reorders the points
on cosets of the subgroup  $\Bbb Z$ (i.e., on the trajectories of the
odometer) in a quite complicated way.

The analysis below is similar to that made in
\cite{Mo} in a simpler case;
namely, a detailed comparison of the ordinary order with the
{\it Morse order} arising from the study of the Morse automorphism.

\medskip
It is easy to deduce from the definition of the Pascal automorphism
that it is given by the following formula:
$$x\mapsto Px;\quad  P(0^m1^k\textbf{10}**)=1^k0^m\textbf{01}**, \quad m,k=0,1 \dots.$$
It is convenient to write the automorphism $P^{-1}$ in a similar form:
 $$P^{-1}(1^k 0^m\textbf{01}**)=0^m1^k\textbf{10}**,\quad m,k=0,1 \dots.$$

The passage from $P$ to $P^{-1}$ swaps $m$ and $k$, i.e., $0$ and $1$.
The automorphism $P$ and its inverse $P^{-1}$ are defined for all $x$
with infinitely many zeros and ones, i.e., on the set
$\textbf{Z}_2\setminus\Bbb Z$. On the other hand, since
$P(x)$ lies in the same coset as $x$, one can ask what is the difference
$P(x)-x$. We summarize the answer in the following lemma.

\begin{lemma}
The Pascal automorphism is given by the formula
 $$P(0^m1^k10**)=1^k0^{m+1}1**, \quad m,k\geq 0;$$
if we consider argument as rational integers then that formula is the following:
$$P(2^{m+k+1}-2^m+r)=2^{m+k+1}+2^k-1+r, \quad m,k\geq 0, \;\forall r
\in\operatorname{Ker} (\theta_{m+k+1});
$$
here $\theta_n$ is the homomorphism defined by the formula
$\theta_n:\textbf{Z}_2 \rightarrow \textbf{Z}_2 /Z_{2^n}$, and its
kernel consists of the sequences with zeros at the first $n$ positions.

Correspondingly, the time change that transforms the
odometer $T$ into the Pascal automorphism $P$ is as follows:
$$Px=T^{n(x)}x\equiv x+n(x),$$
where
 $$ n(x)=P(x)-x=1^{\min(k,m)}0^{|m-k|}10^{\infty},$$
or in numerical form:
$$n(2^{m+k+1}-2^m)=2^m+2^k-1.$$
\end{lemma}

These formulas define the Pascal automorphism for any element
$x \in \textbf{Z}_2$ whose expansion contains a fragment of the form $10$,
i.e., for any dyadic integer that is not of the form
$0^m1^{\infty}$, $m \geq 0$ (a negative integer) and not of the form $1^k0^{\infty}$ (a positive integer).
The above formulas for
$n(x)$ are easy to verify in either of the cases
$m>k$ and $m\leq k$.

Another description of the value $n(x)$ is the following:
$$n(x)=min\{k>0: c(x,T^x)=0\},$$
in other words, first moment when number of $0$-s (and $1$-s) becomes in $T^kx, k>0$ the same as in $x$.

\textbf{Remark.} $$\int n(x)dm(x)=\infty;$$ (here $dm$ is Haar measure),
because accordingly to the Belinskaya theorem (\cite{Be}) the finiteness of integral means that $n(x)=const$.
But it is easy to make direct calculations of integral.

It is easy to reinterpret Pascal automorphism as piecewise map of the unit interval with Lebesgue measure;
formulas will follow from formulas above.

Now consider the functions
$n_k(x)$ defined as $P^k x=T^{n_k(x)}x=n_k(x)+x$ for all positive integers $k=0,1,2 \dots$:
$$n_0(x)=0,\quad n_1(x)=n(x), \quad  n_k(x)= P^k x-x, \quad \dots.$$
A formula for $n_k(x)$ follows from the definition, as
described in the lemma below.

\begin{lemma}The function $n_k(x)$ is determined by the following formula:
$$n_{k+1}=n_1(n_k(x)+x)+n_k(x).$$
\end{lemma}

\begin{proof}
We have $P^{k+1}x \equiv T^{n_{k+1}(x)}x=n_{k+1}(x)+x$.

But
$$P^{k+1}x=P(P^k)x=T^{n_1(P^kx)}P^kx,$$
so that $$P^{k+1}x=T^{n_1(P^kx)+n_k(x)}x,$$
i.e., $x+n_{k+1}(x)=n_1(P^kx)+n_k(x)+x=n_1(x+n_k(x))+n_k(x)+x$.
\end{proof}

Thus
\begin{eqnarray*}
P^k x&=& n_1(n_{k-1}(x)+x) + n_{k-1}(x)=\dots \\&=& x + n_1(x) +
n_1(n_1(x)+x) +{\ldots} +n_1(n_{k-1}(x)+x).
\end{eqnarray*}

Observe that the formulas expressing the functions
$n_k(x)$, $k>1$, in terms of the function $n_1(x)=n(x)$ are, of
course, universal: they hold for a time change in an arbitrary
automorphism. We will use only the function
$n(\cdot)=n_1(\cdot)$.

The special
orbit $\Bbb Z \subset \textbf{Z}_2 $ of the odometer
breaks into countably
many bounded orbits of the Pascal automorphism
with finitely many zeros and ones;
namely, every positive integer $x\in \Bbb N$ belongs to the semiinfinite
orbit with the endpoint $2^s-1$ where $s$ is the number of ones in the
binary expansion of $x$; and every negative integer belongs to the
semiinfinite
orbit with the initial point $-2^s+1$. All the other points of $P$,
regarded as linearly ordered sets, are of order type
$\Bbb Z$.

In connection with the formula for $Px$, an important question arises
(we have already mentioned it above):
how do the cosets of the subgroup
$\Bbb Z$, i.e., the orbits of the odometer, transform under the action
of the Pascal automorphism? In other words, how can one describe the
map $k \mapsto n(x+k)$ as a substitution in the group $\Bbb Z$:
$$ \sigma_x:t\mapsto n(x+t), \quad t\in \Bbb Z ?$$

Thus the Pascal automorphism determines a random (the
randomness parameter is  $x\in \textbf{Z}_2$) infinite substitution $\sigma_x$
that maps $\Bbb Z$ (simply as a countable set) to itself and has
infinitely many infinite cycles. Thus we obtain a measure
on the group $S^{\Bbb Z}$ of all infinite substitutions of $\Bbb Z$
which is the image of the Bernoulli measure under the correspondence
$\textbf{Z}_2 \ni x \mapsto \sigma_x\in S^{\infty}$. This measure
differs substantially from the measure arising in the similar analysis
of the Morse transformation
\cite{Mo} (in that case, the measure is supported by one-cycle
substitutions);
its study is of the main interest (see below).

A general principle says that every time change in a dynamical system
with invariant measure determines a measure on the group of infinite
substitutions of the group (time), and the properties of this measure
allow one to derive conclusions about the new system. It is in this
sense that we used the thesis that an action of a group with
invariant measure can be regarded as an action of a random
substitution on this group. But for this we should choose
a {\it reference action}, an initial action in which one makes a time change and which
orders the trajectories; then a random substitution corresponds to a
translation on the group. In our case, this reference action is the odometer.

\section{The stationary model of the Pascal automorphism, and the
Pascal measure}

\subsection{Definition of Pascal Measure}

Now let us represent the action of the Pascal automorphism in more
traditional terms, as the shift in the space of two-sided sequences of
zeros and ones (i.e., again the infinite-dimensional cube) equipped
with some shift-invariant measure. It is this representation
that will be used in the main theorem.

The following theorem determines a stationary model of the Pascal
automorphism.

\begin{theorem}
Given $x\in \textbf{Z}_2$, define a sequence
$y_n(x)$ of zeros and ones as follows:
$$y_n=(P^nx)_1, \quad n \in \Bbb Z;$$
here $(\cdot)_1$ is the first digit of the dyadic number in the
parentheses, i.e., $0$ or~$1$. Thus we have a map
$$S:\prod_1^{\infty}\{0,1\}\equiv\textbf{Z}_2 \rightarrow
Y=\textbf{2}^{\Bbb Z}=\prod_{-\infty}^{\infty}\{0,1\}$$
given by the formula
  $$\textbf{Z}_2 \ni x \mapsto y\equiv \{y_n\}_{n\in \Bbb Z}:\quad
  y_n(x)=(P^nx)_1,\; n\in \Bbb Z.$$

The map $S$ sends the Bernoulli measure $\nu$ on the
infinite-dimensional cube to some measure $S_*\nu\equiv \pi$ on another
(bilaterally infinite-dimensional) cube
$Y=\prod_{-\infty}^{\infty}\{0,1\}$, which we will call the {\it
Pascal measure}. Thus $S$ is an isomorphism of measure spaces that
sends the Pascal automorphism $P$ of the space
$\textbf{Z}_2$ to the two-sided shift in the space
$Y=2^{\Bbb Z}$ with stationary measure $\pi= S_*\nu$.
 \end{theorem}

In terms of paths in the Pascal graph, the map $S$ can be described as
follows. Given such a
path, regarded as a sequence of vertices in the Pascal graph,
this  map associates with it  the
sequence of changes of the first segment in the course of the adic
evolution of the path.

\begin{lemma}
The partition of the space $\textbf{Z}_2$ into two sets according to
the value of the first coordinate is a (one-sided) generator of the
Pascal automorphism. In other words, almost every point is uniquely
determined by  the sequence of the first coordinates of
its images under the action of the positive powers of the Pascal
automorphism:
$$x \leftrightarrow \{(P^nx)_1\}_{n \in \Bbb N}$$
is a bijection for almost all $x \in \textbf{Z}_2$. That is,
the map $S$ is an isomorphism of the space
$\textbf{Z}_2$ onto the image of $\prod_{n \in \Bbb N}\{0,1\}$ under $S$.
\end{lemma}

\begin{proof}
It can easily be seen from the formulas that two distinct elements of
the group $\textbf{Z}_2$ whose first noncoinciding digits have index $n$
lead to sequences
$(P^kx)_1$ that have at least one noncoinciding digit with index less than
$2^n$.
\end{proof}

\medskip\noindent
{\bf Remarks. 1.} For the odometer, the same partition according to
the  first coordinate is obviously not a generator.

\noindent
{\bf2.} The support of the Pascal measure is of great interest
(see the table below).
In \cite{Me} (see also \cite{MM}) it is proved that the number of
cylinders in the image (the ``complexity of the Pascal
automorphism''), i.e., the number of words of length $n$, is
asymptotically equal to  $n^3/6$.

\bigskip
In order to study the Pascal
automorphism, it is convenient to encode dyadic integers (elements of
$\textbf{Z}_2$) not by binary sequences of zeros and ones, as above,
but  by sequences of pairs of positive integers
$(m_i(x), k_i(x))$ defined as follows:
$$x=(0^{m_1(x)}1^{k_1(x)} \textbf{10} 0^{m_2(x)}1^{k_2(x)} \textbf{10} 0^{m_3(x)}1^{k_3(x)}\dots).$$
The numbers $m_i(x)$ (respectively, $k_i(x)$) are the lengths of the words
of zeros (respectively, ones) between two
(the $(i-1)$th and the  $i$th) occurrences of the word $\textbf{10}$
in the binary expansion of the number $x$. The coordinates $(m_i(x), k_i(x))$
will be temporarily called the
{\it pair coordinates}.
 Clearly, the ordinary
coordinates can be recovered from them in a trivial way.

Obviously, the vectors $(m_i(x),k_i(x))$ form a sequence (in $i$) of
independent identically distributed random vectors (regarded as
functions of $x$), and the distribution of each pair is as follows:
$${\rm Pr}\{m_i=m,k_i=k\}=2^{-(m+k-2)}, \quad m,k=0,1,2 \dots.$$

The map
$S:\textbf{Z}_2 \rightarrow Y=\prod_{-\infty}^{\infty}\{0,1\}$ defined
above can be
written in a more specific form. This leads to the notion of {\it
supporting words}.

Assume that the first pair coordinate of an element
$x \in \textbf{Z}_2$ is  $m_1(x)=m\geq 0$, $k_1(x)=k\geq 0$,
i.e., the initial segment looks as $$x=0^m1^k01**.$$
Consider the elements $x,Px,P^2x, \dots P^sx$, with $s=C_{m+k+1}^m$, and
write down the first coordinates of these elements. We will obtain a
word of length $s$, which, by definition, is the beginning of the
$S$-image  of $x$. The following inductive rule for the pair
coordinates can be verified straightforwardly.

\begin{lemma}
If $k>0$, then the first pair coordinate of
$Px$ is $m_1(Px)=0$, and $(Sx)_1=S(Px)_1=S(x)_2=\dots
=S(P^{k-1}x)_1=S(x)_{k-1}=1$.

If $m>0$, then $m_1(P^sx)=m+1$.

If, additionally, $m_2(x)>0$, then $k_1(P^sx)=k$; and if $m_2(x)=0$, then $k_1(P^sx)=k+k_2-1$.
\end{lemma}

The supporting word $O_{m,k}$ is exactly the word of length $s$ that
consists of the first coordinates of the elements
$x,Px, \dots, P^sx$.

\begin{corollary}
Almost every sequence with respect to the Pascal measure
$\pi$ contains an arbitrarily long series of ones.
\end{corollary}

Indeed, almost every element $x\in \textbf{Z}_2$ contains infinitely many
pair coordinates
$m_i(x),k_i(x)$ with arbitrarily large values, and, by the first claim
of the lemma, the number of consecutive ones in the image is equal to
$k_i(x)-1$.

The other two claims of the lemma allow us to successively construct
the coordinates of $Sx$ regarded as an element of
$Y=\prod_{-\infty}^{\infty}\{0,1\}$. The image $Sx$ is the
concatenation of the random supporting words determined by the
sequence of independent pair coordinates
$(m_i(x),k_i(x))$. The construction of the next supporting word is
defined by the rules described in the lemma. The lengths of supporting
words grow exponentially with $m$ and $k$, and the study of the Pascal
measure reduces to the study of the sequence of supporting words.
Observe that the first coordinate $m$ changes deterministically,
$m\mapsto m+1$, while the second coordinate $k$ grows randomly with mean value
$1$. The statistics of words can be studied in terms of the asymptotic
properties of the growing supporting words; each of these words is
deterministic, but the choice of words is random.

\medskip\noindent{\bf Remark.} Above we have discussed the coordinates
$Sx$ with positive indices. In order to understand the structure of
coordinates with negative indices, one should, according to the remark
above, consider an
element $x$ in exactly the same way, but swap $0$ and $1$ in the
definition of the pair coordinates. On the other hand, it is not
difficult to prove that the coordinates with negative indices are
uniquely determined by the coordinates with positive indices for almost all points with
respect to the Pascal measure.
\medskip

Let us consider an example of the dynamics of the Pascal automorphism and
find the corresponding supporting word.

Consider the fragment of a trajectory of the Pascal
automorphism with
$$m_1=3,\;k_1=3,\; x=000111\textbf{10}.$$
The length of this fragment is equal to $C_{m_1+k_1+1}^{k_1}+1=C_7^3=35+1=36$:
$x\to P(x)=11100001\to P^2(x)\to\dots\to P^{35}=00001111$
(recall that we should find the first occurrence of the word
\textbf{10} and then apply the algorithm described above). We arrange the 36
successive images of the point $x$ in a $6\times 6$ table:
$$
\begin{array}{cccccc}
  00011110 & 11100001 & 11010001 & 10110001 & 01110001 & 11001001 \\
  10101001 & 01101001 & 10011001 & 01011001 & 00111001 &11000101 \\
  10100101 & 01100101 & 10011001 & 01010101 & 00110101 & 10001101\\
  01001101 & 00101101 & 00011101 & 11000011 & 10100011 & 10001101 \\
  10010011 & 01010011 & 00011101 & 10001011 & 01001011 & 00101011 \\
  00011011 & 10000111 & 01000111 & 00100111 & 00010111 & 00001111
\end{array}
$$

The sequence of the first digits in this fragment of the trajectory,
i.e., the supporting word
$O(3,3)$, is as follows:
$$(1,1,1,0,1,1,0,1,0,0,1,1,0,1,0,0,1,0,0,0,1,1,0,1,0,0,1,0,0,0,1,0,0,0,0).$$

The data in Table~1 shed some light on the distribution of
probabilities of cylinders of
lengths 6 to 10 with respect to the Pascal measure.\footnote{These computations were performed on my
request by the PhD students I.~Manaev and A.~Minabutdinov.}

\begin{table}
\caption{}
\begin{tabular}{|l|l|l|l|}
\hline
{\small Length of the word} & {\small Number of groups} & {\small
Cardinality of the groups}
&{\small Measure of the groups}\\
\hline
6 & 3   & 5 & 0,312484741\\
   &   &  12 & 0,374969482\\
 & & 20 & 0,312213898\\
\hline
7&  4&  2&  0,124998093\\
&   &   14& 0,437492371\\
&   &   16& 0,249990463\\
&   &   24& 0,187408447\\
\hline
8&  5&  1&  0,062498093\\
&   &   12& 0,374994278\\
&   &   19& 0,296865463\\
&   &   20& 0,156238556\\
&   &28&    0,109274864\\
\hline
9&  5&  10& 0,312494278\\
&   &   22& 0,343740463\\
&   &   24& 0,187488556\\
&   &   24& 0,093736649\\
&   &   32& 0,062393188\\
\hline
10& 6&8&    0,249994278\\
&   &   21& 0,328117371\\
&   &   28& 0,218738556\\
&   &   29& 0,113267899\\
&   &   28& 0,054672241\\
&   &   36& 0,03504467\\
\hline
\end{tabular}
\end{table}

\bigskip

One can observe that the cylinders are divided into several groups
such that inside each group the probabilities are equal. This
simplifies obtaining lower bounds on the growth of the scaling entropy,
which, according to the original plan (\cite{V}) described below,
should allow one to prove that the
spectrum of the Pascal automorphism is continuous.

\subsection{Pascal automorphism and $\sigma$-finite invariant measure}

It happened that Pascal automorphism (without that name) had been appeared in the paper by
A.Hajian,Y.Ito and S.Kakutani in 1972 (\cite{HIK}) and S.Kakutani in 1976 (\cite{Ka} \footnote{I am grateful to Professor A.Hajian who gave me this
references when I had visited NE-University in April 2011}. Namely
the authors of \cite{HIK} used a general Bernoulli (product) measure on the space $\textbf{Z}_2$ (instead of Haar measure)
which is quasi-invariant measure with respect to odometer $T$. Using Radon-Nikodim cocycle authors constructed the new automorphism
$R$ of the space $\textbf{Z}_2\times \Bbb Z$ with $R$-invariant $\sigma$-finite measure.
The automorphism $R$ is a skew-product and it induces on the initial space  $\textbf{Z}_2$ exactly Pascal automorphism,
which preserves Bernoulli measure (see remark above). The corresponding ceiling function for induced automorphism at point $x$
is just our function $n(x)$ --- remember that $\int n(x)dm(x)=\infty$ so global measure is infinite.
Ergodicity of $R$ follows from ergodicity of Pascal automorphism or from $0-1$-Law of the "exchange-$\sigma$-fields".
The automorphism had been used a tool for construction of some example. In \cite{Ka}  automorphism of unit interval was
defined with formula $P(0^m1^k10*)=1^k0^{m+1}1*$ was used for some statistical question.

From our point of view this construction can be considered as example of generalized random walks on the group $\Bbb Z$,
or even more general construction. Together with other examples concerning to the $\sigma$-finite measure we will consider
the topic elsewhere.

\section{Criteria for the continuity of the spectrum of an
automorphism}

We formulate now several plans of the proof of the fact that the spectrum of the
Pascal automorphism is purely continuous; more exactly, that in the
orthogonal complement of the constants in the space
$L^2(I^{\infty},\mu)$, the spectrum of the unitary operator
$U_P$ associated to the Pascal automorphism by the formula
$U_Pf(x)=f(Px)$ is continuous. This problem, along with the definition
of the Pascal automorphism, was suggested by the author \cite{V} in 1980
and subsequently considered in a series of papers
(e.g., \cite{PS,Me,MP,Ja,JaR}),
where various useful properties of the Pascal automorphism were
studied; however, the problem has not been solved up to now.

 \subsection{Entropy approach}
The original plan for solving this problem was related to the scaling
entropy, average metrics, etc.\
(see \cite{V,{MM}}). The method suggested by the author for proving the continuity
of the spectrum  (see \cite{V}) was based on the following fact: the
spectrum is purely continuous if and only if
the result of averaging an arbitrary semimetric along its trajectory
is a trivial (constant) metric, i.e.,
$$\lim_{n \to \infty} n^{-1}\sum_{s=0}^{n-1}\rho(P^sx, P^sx')={\rm
const}$$
for almost all pairs $x,x'$ with respect to the measure $\mu\times \mu$.

In particular, in \cite{V} the following theorem is proved. Let $T$ be
an automorphism of a Lebesgue space $(X,\mu)$. The
spectrum of the corresponding unitary operator $U_T$ in $L^2(X,\mu)$
is purely
continuous in the orthogonal complement of the constants
if and only if for every admissible semimetric for
which the limiting average metric is also admissible, the scaling sequence
for the entropy is bounded.

The stipulation about the admissibility of the limiting average metric
is superfluous, because it holds for every admissible initial metric,
as was shown by F.~Petrov and P.~Zatiskiy (see \cite{V}). Originally,
the admissibility of the limiting metric was proved for the class of
compact
and bounded admissible semimetrics. On the other hand, in \cite{V}
it is proved that the average semimetric is constant if and only if
the scaling sequence for the entropy is unbounded, i.e., the
$\varepsilon$-entropy of the spaces obtained by successive averagings
tends to infinity. For the Hamming metric, a close result was earlier
proved in \cite{F}. Thus one might prove the continuity of the spectrum
by bounding the growth of the entropy of the
prelimit average metrics from below by some growing sequence. Moreover, it would
suffice to do this only for cut semimetrics, which have always been
regarded in
ergodic theory not as metrics, but rather as generating partitions.
Recall that a {\it cut semimetric} is a semimetric determined by a
finite partition of a measure space into measurable subsets
$\{A_i\}$, $i=1, \dots ,k$, by the following formula:
$\rho(x,y)=\delta_i(x)i(y)$, where $i(x)$ is the index of the set
$A_i$ that contains  $x$.

Summarizing, we can formulate the following criterion for the
continuity of the spectrum of an
automorphism.

\begin{theorem}
Given an automorphism $T$, the spectrum of the operator
$U_T$ in the orthogonal complement of the constants is continuous if
and only if any of the following equivalent conditions is satisfied:
\begin{enumerate}
\item[{\rm1.}] For an arbitrary cut metric
$\rho$, the limit of the average metrics  is a constant metric:
  $$\lim_n n^{-1}\sum_{k=0}^{n-1}\rho(T^kx,T^ky)={\rm const}
  \quad\mbox{a.e.}$$

\item[{\rm2.}]  For an arbitrary initial semimetric,
the $\varepsilon$-entropy of its averages
grows unboundedly.
\end{enumerate}
\end{theorem}

We emphasize that the partition that determines a cut metric in this
theorem is not at all assumed to be a generator.

Note also that there is a connection between the above assertion about the
scaling entropy and the criterion for the discreteness of the spectrum found by
A.~Kushnirenko \cite{Ku} using A-entropies.

However,  this plan has not yet been realized
because of the difficulties involved in
bounding the entropy (see the remark to the tables above). That is why
here we use other considerations, only partially related to
the ones discussed above.

The realization of this plan is in the project and some problems of it did not overcome up to now.

\subsection{Time changing and locally finite permutations}

\begin{lemma}
1.The shift $S$ in the space $A^{\Bbb Z}$ of sequences in a finite
alphabet $A=\{1,2,\dots ,l\}$ with stationary (shift-invariant) measure
$\mu$ has a pure point spectrum if and only if
almost all realizations $\{x_n\}\in A^{\Bbb Z}$ are (Bezicovitch) almost
periodic $A$-valued functions on
$\Bbb Z$.

2.The automorphism has pure continuous spectrum (in orthogonal compliment to the constants) iff
for any its realizations as the shift in the space of bounded functions on $\Bbb Z$ with values in $\{0;1\}$ (for any binary generators) 
the measure of the set of all almost periodic (by Bezicovitch) functions equal to zero, 

or,

equivalently, 

3.For each binary generators there exists constant $c$ such that the limit of Hamming distance ($\lim_n \frac{1}{n}\sum_{i=0}^n x_i-y_i$ 
between $\mu$-almost all pairs of the elements $x$ and $y$ of the space $l^{\infty}({\Bbb Z};\{0,1\})$ is greater than $c$.

\end{lemma}

This gives the following procedure: choose an arbitrary points  $x$ and $y$ in $l^{\infty}({\Bbb Z};\{0,1\})$,
and find the uniform estimation of the quantity $\hat \rho (x,P^k x)=\lim_n \frac{1}{n}\sum_{k=0}^{n-1} \rho(T^kx,T^ky)$. 
Because we had generated the orbits of Pascal automorphism as sequential permutations of the orbit of odometer ("time changing"),
it is enough to prove that quotients of those permutations for points $x,y$ asymptotically separated from identity. 

The corresponding permutations were considered for Morse automorphism in \cite{Mo} and was called "locally finite permutation" of $\Bbb Z$.
In fact we already introduced those permutation in the previous paragraph and the question is to find asymptotical behavior of that
permutations for different points. We will describe it in the next paragraph.
 
Translated by N.~V.~Tsilevich.

\begin{thebibliography}{99}

\bibitem{F} S.~Ferenczi, Measure-theoretic complexity of ergodic systems.
\emph{Israel Math. J.} \textbf{100}
(1997), 180--207.

\bibitem{Ja} \'E.~Janvresse and T.~de la Rue, The Pascal adic transformation is loosely Bernoulli,
\emph{Ann. Inst. H. Poincar\'e (B), Probabilit\'es et Statistiques},
\textbf{40} (2004), No.~3,
133--139.

\bibitem{JaR}  \'E.~Janvresse, T.~de la Rue, and Y.~Velenik,
Self-similar corrections to the ergodic theorem for the Pascal-adic
transformation, \emph{Stochastics and Dynamics}  \textbf{5} (2005),
No.~1, 1--25.

\bibitem{Ku} A.~Kushnirenko,
 Metric invariants of entropy type. {\it Uspekhi Mat. Nauk} {\bf 22}  (1967),
 No.~5 (137), 57--65.

\bibitem{MM} A.~A.~Lodkin, I.~E.~Manaev, and A.~R.~Minabutdinov,
Asymptotic behavior of the scaling sequence of the Pascal adic
transformation. {\it Zapiski Nauchn. Semin. POMI} {\bf 378} (2010),
58--72. English translation to appear in {\it J. Math. Sci.}

\bibitem{Me} X.~Mela, Dynamical properties of the Pascal adic and related systems. PhD Thesis,
University of North
Carolina at Chapel Hill, 2002.

\bibitem{MP} X.~Mela and K.~Petersen,  Dynamical properties of the Pascal adic transformation,  \emph{Ergodic
Theory Dynam. Systems} \textbf{25} (2005), 227--256.

\bibitem{PS} K.~Petersen and  K.~Schmidt,
Symmetric Gibbs measures. \emph{Trans. Amer. Math. Soc.} \textbf{349} (1997),
No.~7, 2775--2811.

\bibitem{V81} A.~M.~Vershik,
 Uniform algebraic approximations of shift and multiplication operators.
{\it Sov. Math. Dokl.} {\bf24} (1981), 97--100.

\bibitem{V82} A.~M.~Vershik,
A theorem on periodical Markov approximation in ergodic theory.
 {\it J. Sov. Math.} {\bf28} (1985), 667--674.
 Another English translation in: {\it Ergodic Theory and Related Topics}
 (Vitte, 1981). {\it Math. Res.} {\bf12}. Akademie-Verlag, Berlin,
 1982, pp.~195--206.

\bibitem{V4} A.~Vershik, Dynamics of metrics in measure spaces and their asymptotic invariants.
{\it Markov Process. Related Fields} {\bf16} (2010), No.~1, 169--185.

\bibitem{Mo}A.~Vershik. Orbit theory, locally finite permutations and Morse
arithmetic.
In: {\it Dynamical Numbers: Interplay between Dynamical Systems and
Number Theory},
Contemp. Math. {\bf532} (2010), pp.~115--136.

\bibitem{V} A.~M.~Vershik,
Scaling entropy and automorphisms with purely point spectrum.
{\it Algebra i Analiz} {\bf 23} (2011),  No.~1.
English translation to appear in {\it St.~Petersburg Math. J.}

\bibitem{Ve}A.~M.~Vershik and S.~V.~Kerov,
Combinatorial theory and the K-functor. {\it J. Sov. Math.} {\bf38}
(1987), 1701--1733.

\bibitem{HIK}A.~Hajan., Y.~Ito, S.Kakutani.
Invariant Measure and Orbits of Dissipative Transformations.
{\it Adv. in Math.} {\bf 9} (1972), No ~1, pp.~52-65.

\bibitem{Ka} S.~Kakutani.
A Problem of Equidistribution on the Unit Interval $[0,1]$.
{\it Springer Lecture Notes in Math.}, {\bf 541}, (1976), pp ~369-375.

\bibitem{Be} R.~Belinskaya. Orbit partitions of Lebesgue space of ergodic automorphisms. Funct.Anal and Its Appl,
v.2 No 3.,4-16, (1968).
\end
{thebibliography}
\end{document}